\documentclass[12pt]{amsart}

\usepackage{amsfonts}
\usepackage{amssymb,enumerate}
\usepackage{courier}
\usepackage{graphicx}
\usepackage[all]{xy}
\usepackage{mathtools}

\usepackage{caption}
\usepackage{subcaption}

\newtheorem{lem}{Lemma}[section]
\newtheorem{cor}[lem]{Corollary}
\newtheorem{prop}[lem]{Proposition}
\newtheorem{thm}[lem]{Theorem}

\newtheorem{defin}[lem]{Definition}
\newtheorem{Ex}[lem]{Example}
\newtheorem{Remark}[lem]{Remark}
\newtheorem{Construction}[lem]{Construction}
\newtheorem{Notation}[lem]{Notation}
\newtheorem{Fact}[lem]{Fact}
\newtheorem{Discussion}[lem]{Discussion}
\newtheorem{Notationdefinition}[lem]{Definition/Notation}
\newtheorem{Remarkdefinition}[lem]{Remark/Definition}

\newenvironment{remark}{\begin{Remark}\rm}{\end{Remark}}

\newcommand{\PP}{\mathbb{P}}
\newcommand{\RR}{\mathbb{R}}

\newcommand{\ZZ}{\mathbb{Z}}

\newcommand{\Ind}{\mathrm{Ind}}
\newcommand{\Inf}{\mathrm{Inf}}
\newcommand{\GST}{\mathrm{GST}}
\newcommand{\rank}{\mathrm{rk}}
\newcommand{\image}{\mathrm{im}}

\newcommand{\idots}{\reflectbox{$\ddots$}}

\begin{document}

\bibliographystyle{amsplain}

\title[The structure of symmetric n-player games]{The structure of symmetric n-player games when influence and independence collide}

\author[Steel]{Mike Steel}
\author[Taylor]{Amelia Taylor}

\date{\today}


\begin{abstract}
We study the mathematical properties of probabilistic processes in
which the independent actions of $n$ players (`causes') can influence
the outcome of each player (`effects').   In such a setting,  each
pair of outcomes will generally be statistically correlated, even if
the actions of all the players provide a complete causal description
of the players' outcomes, and even if we condition on the outcome of
any one player's action.  This correlation always holds  when $n=2$,  but
when $n=3$ there exists a highly symmetric process, recently studied,
in which each cause can influence each effect, and yet each pair of
effects is probabilistically independent (even upon conditioning on any one
cause).  We study such symmetric processes in more detail, obtaining a
complete classification for all $n \geq 3$.  Using a
variety of mathematical techniques, we describe the geometry and
topology of the underlying probability space that allows independence
and influence to coexist. 
\end{abstract}

\keywords{Conditional independence, causality, quadratic form, homology}

\maketitle

\section{Introduction}

The study of causality is a long-standing topic at the interface of
statistics and the philosophy of science. It is also an area where
the mathematical analysis of graphical models has led to some
important recent advances (see e.g. \cite{hof, pearl}). In this paper, we
investigate a particular class of symmetric causal processes which
achieves two apparently conflicting requirements (`independence' and `influence' defined shortly).

In Section~\ref{sec:background}, we provide formal definitions, but give
the main ideas here to facilitate the discussion.   Let $E_1, \ldots E_n$ be $n$ dichotomous (two states) random variables
with the same state spaces, which we call `effects' and let $C_1, \ldots
C_n$ be $n$ independent dichotomous random variables, also with the
same state spaces, which we call `causes'.   

We say that a  cause
$C_i$ `influences' effect $E_j$ if there exists at least one
assignment of states for the remaining causes such that a change in
the state of $C_i$ changes the (conditional) probability of at least
one state of the $E_j$ \cite{SS}.   `Independence' refers to pair-wise
probabilistic independence of the effects (either absolutely, or conditional on
knowing the state of any one cause). 

We explore a symmetric system because it is applicable to any scenario
in which the probability of $E_i$ depends just on how many causes
take the same value as $C_i$. We can view this process as a game where
we identify $C_i$ with the action of some player $i$;  the
outcome for each player $i$ (i.e. $E_i$) then depends solely on how many of
the other players chose the same action.  

For example, suppose there are $n$ flowering plants in an area of
study.  For plant $i$, the cause $C_i$ might describe whether the plant
flowers early or late.  The corresponding effect $E_i$ could denote
whether or not a plant is 
pollinated.  For example, flowering early with many other flowers might
be advantageous because such a mass flowering attracts more bees and increases
the probability the plant is pollinated.  On the other hand, there may
be a limit in the number of bees,  so flowering at the same time as a smaller number of
plants may also be advantageous.  Either way, the probability of an
effect (pollination of plant $i$) 
depends on the number of causes which match the cause of that particular
effect (i.e. how many other plants flower at the same time as plant $i$).  

Recently, such processes have been studied in the philosophy of science literature
as they provide insights into the extent to which subsets of causes can render effects independent
(Theorem 5b of \cite{SS}).  The authors of \cite{SS} illustrated such
a process with an entertaining application involving $n$ people playing a tequila drinking game.
In~\cite{SS} they consider just the case $n = 3$.  In the game, the $n$ people
simultaneously and independently reveal a clenched fist or an open
hand (with equal probability), and the states of
the $n$ hands are regarded as the $n$ causes.  The event that person $i$ drinks
tequila is $E_i$, for $1\leq i\leq n$.  The rules for determining if
person $E_i$ drinks when $n = 3$ are that if a player's hand position
is unique then they drink with probability $p_1= 1$.  For the ties (e.g. a
tie of two or three), those in the tie drink independently with
probability $p_2=1/2$ when there are two people in the tie and probability
$p_3=1/3$ when there are three people in the tie (see Fig.~\ref{hands}).  The probabilities used
here are quite special when we consider influence and independence in
relation to each other and the effect on the system.  We study what is
special and how it can generalize.  We call this extension of this game  to $n$ players the `generalized
symmetric tequila problem' (GST) but, as noted in the previous paragraph, the relevance of such processes extends well 
beyond bar drinking games.

\begin{figure}[ht]
\begin{center}
\includegraphics[scale=.55]{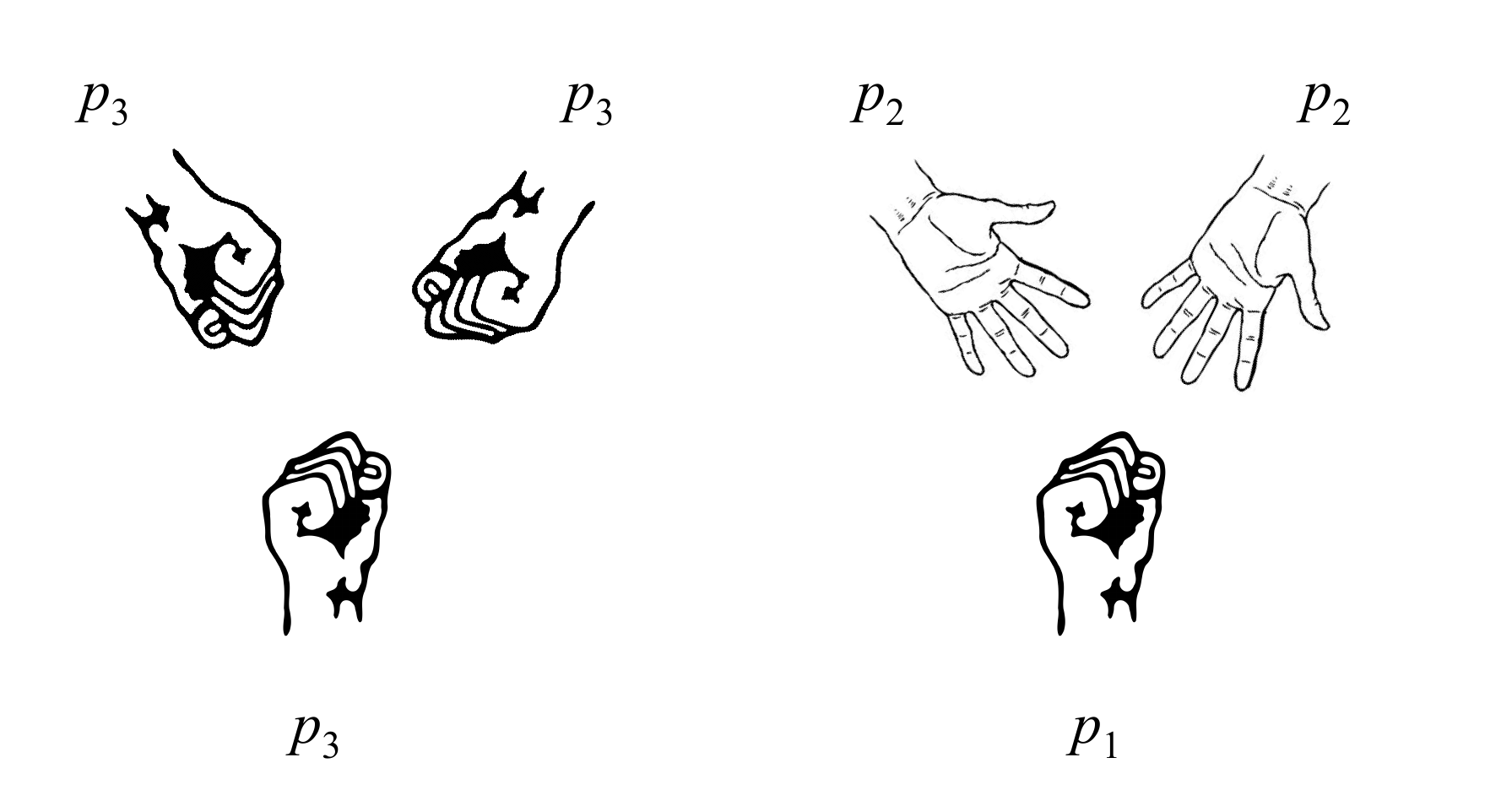}
\end{center}\centering
\caption{A simple three-player game exhibiting independence and influence, for  various values of $(p_1, p_2, p_3)$; including 
$(1,\frac{1}{2}, \frac{1}{3})$ from \cite{SS}, and $(3^{-1}, 3^{-2}, 3^{-3})$ from Section~\ref{explicit2}; see text for details.}
\label{hands}
\end{figure}

Our main results assume the system has some symmetry, as explained at
the beginning of Section~\ref{sec:polynomials} and we define three 
spaces in this context: $\Inf_n$, $\Ind_n$ and $\GST_n = \Inf_n\cap
\Ind_n$. These spaces are formally defined in
Section~\ref{sec:polynomials} but, in short, are the set of
probabilities for the fully symmetric system which lead to influence,
independence and both, respectively.   

We fully analyze the case $n = 3$ (Section~\ref{sec:polynomials}), we
establish a useful equivalence relation on $\Ind_n$
(Section~\ref{sec:geometry}), we show that $\Ind_n$ is contractible
(Section~\ref{sec:topology}) but not convex
(Section~\ref{sec:geometry}), and that $\GST_n$ is neither.   

We establish a characterization (Proposition~\ref{prop:influence}) for the
system to be in $\Inf_n$.  We show, via a quadratic form and 
its Hessian matrix, that $\GST_n$ contains infinitely many points for any $n\geq
3$.  We use this structure to investigate the topology and geometry of the space $\GST_n$, with a
main objective being to determine whether or not it is connected.  We show that $\GST_n$
is disconnected for $n = 3, 4$ and connected when $n\geq 8$ in 
Theorem~\ref{gstConnected} and Corollary~\ref{cor:connected}. The remaining cases where $n=5,6,7$ seems an  interesting 
question for future study.  

Our results involve an interplay of
linear algebra, analysis, combinatorics and topology, including some
classical results in these fields, such as Sylvester's Inertia
Theorem, Alexander Duality and Smith's theorem on periodic maps.

\section{Formal Setup}\label{sec:background}

We begin by giving the formal set-up of the system of causes and
effects, and proceed to provide formal definitions of influence,
and conditional independence.   

Let $E_1, \ldots, E_n$ and $C_1, \ldots, C_n$ be random variables with
two possible states (also called `dichotomous'), labeled throughout
this paper as $0$ and $1$ (although our results do not depend on this).  
We assume that the $C_i$ are
(mutually) independent, and each event $E_j$ depend on the outcome of the
events $C_i$; accordingly we call the $C_i$ {\em causes} and the $E_j$
{\em effects}.  To simplify notation, we write conditional probabilities of the form
 $\PP(E_i=1|*)$ more simply as $\PP(E_i| *)$ (i.e. $E_i=1$ is the event that $E_i$ `occurs'). 
 The model we study makes the following assumptions: 

\bigskip

\begin{itemize}

\item[{\bf (A1)}] The causes are (mutually) independent, with 
$\PP(C_i = 1) =r$  for some $0<r<1$.

\item[{\bf (A2)}] The effects are conditionally independent, given the joint outcome of the causes.
\begin{equation}\label{eq:probE}
\PP(E_i|  \bigwedge_{j=1}^n C_j = x_j) =  \begin{cases} p_k, & x_i = 0, \text{ and } k \text{
    total causes are in state } 0;\\
q_k, & x_i = 1, \text{ and } k \text{ total causes are
  in state } 1.\end{cases}\notag
\end{equation}

\end{itemize}

\bigskip

Property (A2) states that the probability 
of $E_i$ depends only on the number of causes that are in the same
state as $C_i$.  Flowers often seem to flower with some dependence on
the number of other flowers which have also flowered.  In the tequila
example, $p_1 = q_1 = 1$, $p_2 = q_2 = 1/2$ and $p_3 = q_3 = 1/3$.

In this paper we will mostly deal with the case where $p_k=q_k$ for all $k$, and $r=1/2$ (the fully-symmetric (or GST) model), but
it is helpful to pose the problem more generally.

\subsection{Influence and Independence}

While the set-up we explore has the same number of causes as effects,
we give the definitions here  for arbitrary
numbers of causes and effects.  
\begin{defin}\cite{SS}\label{def:influence}
A set of $k$ causes \emph{influences} a set of $m$ effects if 
for each cause $C_i$, there exists as least one assignment of states
for the remaining $k -1$ causes, such that some change in the state of
$C_i$, while holding the values of the remaining $k-1$ causes fixed,
changes the probability of at least one state of each of the $m$
effects. 
\end{defin}
The influence condition is equivalent to the requirement
that none of the causes can be eliminated for any effect -- that is,
for each $i$,  there is no proper subset $J$ of $\{1, \ldots, k\}$ for which 
$\PP(E_i|\bigwedge_{j=1}^k C_j = x_j)$ can be written 
as a function of $(x_j: j \in J)$, for all $(x_1, \ldots, x_k)$.     

We also study probabilistic independence.  Recall that two random variables
$X$ and $Y$ are independent with respect to a third random
variable $Z$ if and only if $\PP(X \wedge Y | Z ) = \PP(X|Z)\PP(Y|Z)$.  In the
language of causality and graphical models we would say that $Z$
\emph{screens off} $X$ from $Y$.  This language is natural when
looking at graphical models and, to be consistent with that literature,
we will use this phrasing as well. 

The {\em independence condition} is then the  requirement that each cause 
screens off each effect from any other effect.

For example, in the tequila drinking game, any cause $C_k$ screens off
any pair of effects as $\PP(E_i \wedge E_j\mid C_k) = \PP(E_i\mid C_k) \PP(E_j \mid
C_k)$.  However, the reason this example is of interest in ~\cite{SS} is because 
any pair of causes $(C_{k_1}, C_{k_2})$ do not screen $E_i$ from $E_j$
for any pair $(E_i, E_j)$ and yet the set of all three causes
screens off any pair of events.   This provides a contrast to what happens when $n=2$.  
In that case, Theorem 2 of \cite{SS} shows that  neither of two dichotomous causes can screen off $E_1$ from $E_2$ (i.e. the
independence condition fails) whenever the  two  causes:
\begin{itemize}
\item[(a)]   have non-zero joint probability for any combination of states,
\item[(b)] together screen off $E_1$ from $E_2$, and 
\item[(c)] each influence $E_1$ and $E_2$.
\end{itemize}

We might also wonder whether, when $n\geq 3$,  we can strengthen the
independence condition to apply when we condition on more than one
cause. However, there is a limit to the extent to which we can do this if
we wish to also maintain influence, due to the following result, which
follows directly from Corollary 2 of \cite{SS}. 

\begin{prop}
\label{SSresult}
For any model that satisfies (A1), (A2), influence and independence, any
two effects are dependent once we specify the values of  any subset of
the causes of size $n-1$.
\end{prop}

\section{The fully symmetric (GST) model: structure of the probabilities}\label{sec:polynomials}

The symmetric setting where $p_k=q_k$ and $r_k = 1/2$ for all $k \in
\{1,\ldots, n\}$ is of particular interest, as it is tractable and
leads to some interesting results when we couple influence with
independence.  We call the model where $p_k=q_k$ and $r_k = 1/2$ the
{\em generalized symmetric tequila} (GST) setting, as it generalizes
the tequila example in~\cite{SS}, where $n=3$.  We note that
taking $r_k=1/2$ is the natural choice for symmetric games where it is
beneficial to each player play a minority action (for example, if
$p_k=q_k$ is decreasing with $k$), as this provides a Nash equilibrium
strategy. 

We explore the
case $n=3$ further to characterize all the solutions satisfying
influence and independence, before turning to general values of
$n$ as it serves to further understand the example in~\cite{SS}, it
serves as a `boundary' example for larger $n$ and we return to this
example throughout the text.  

Firstly, notice that in the GST setting, $\PP(E_i |C_j = x)$ takes the
same value for each choice of $i, j$ and 
$x$ (this probability is given formally in the proof of
Proposition~\ref{psiprop}).  In particular, $E_i$ and $C_j$ are (pairwise)
independent, for 
any pair $i,j$ (including $i=j$).  
If influence applies then $E_i$ `depends on' $C_j$ (and the other
causes) but this does not translate through to probabilistic independence.

A second basic observation in the GST setting is that symmetry gives the following:
\begin{align}
\PP(E_i) = & \PP(E_i\mid C_k = 0)\PP(C_k = 0) + \PP(E_i \mid C_k =
1)\PP(C_k = 1)\\
 = & \PP(E_i\mid C_k = 0)\frac{1}{2} + \PP(E_i \mid C_k =
0)\frac{1}{2} = \PP(E_i\mid C_k = 0).\notag
\end{align}
Therefore, effects $E_i$ and $E_j$ are independent
if and only if any single cause $C_k$
screens off the two effects.  

In the GST setting,  the conditions (A1) and (A2), coupled with influence
and independence, can be stated more succinctly as:

\begin{itemize}
\item[{\bf (i)}] The causes represent independent tosses of a fair coin;
\item[{\bf (ii)}] The effects are mutually (probabilistically)
  independent once we specify the states of all the causes; 
\item[{\bf (iii)}]  The probability of $E_i$ depends (exactly) on the
  number of causes that take the same value as $C_i$;  
\item[{\bf (iv)}]  Each pair of effects is (probabilistically) independent;
\item[{\bf (v)}]  Each cause can influence each effect.
\end{itemize}

\subsection{The cases $n=2$ and $n=3$}\label{smalln}
In the case where $n=2$, it is easy to verify  that any process that
satisfies properties (i)--(iv) must have $p_1=p_2$ and so must fail
to satisfy the influence condition (v). 

The case where $n=3$ is more interesting. We study independence by studying the
following equation, which follows from direct computation or 
Eqn.~(\ref{eq:gstIndependence}),  assuming $p_k = q_k$ and $r = 1/2$.  
\begin{align}\label{nIs3}
\PP(E_i\mid C_j = 0)^2 = \bigg(\frac{1}{16}\bigg)(p_3 + 2p_2 + p_1)^2 &
= \bigg(\frac{1}{4}\bigg)(p_3^2 + p_2^2 + 2p_2p_1) = \PP(E_i, E_j\mid
C_j =0)\notag\\
\frac{p_1^2}{16} - \frac{p_2p_1}{4} + \frac{p_3p_1}{8} - \frac{3p_3^2}{16} + \frac{p_2p_3}{4} & = 0\notag\\
\frac{1}{16}(p_1-p_3)(p_1-4p_2+3p_3) & = 0
\end{align}
Notice that $p_1 = 1$, $p_2 = 1/2$, $p_3 = 1/3$ is a solution to the
final equation which corresponds to the solution presented for the
original tequila game in \cite{SS}.  Also observe that the space of
probabilities leading to independence consists of two planes, as shown in
Fig.~\ref{fig:n3planes}. 

\begin{figure}[ht]
\begin{center}
\includegraphics[scale=.55]{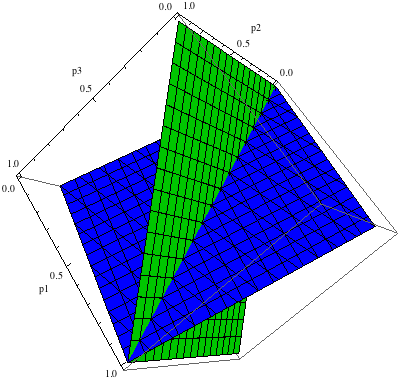}
\end{center}\centering
\caption{The plane $p_1 = p_3$ (green) where influence fails,  
and the plane $p_1-4p_2 + 3p_3 = 0$ (blue) where independence holds.} \label{fig:n3planes}
\end{figure}

Further, any solution with $p_1=p_3$ corresponding to the
vanishing of the first term $(p_1-p_3)$ in Eqn.  (\ref{nIs3})  
fails to satisfy the influence property.   This is an easy example to
work through but also follows from the more general
Proposition~\ref{prop:influence} below.  

The intersection of the two planes is $p_1 = p_2 = p_3$, where
influence clearly fails.  For the remaining points on the plane
$p_1-4p_2+3p_3 = 0$, $p_1\neq p_2\neq p_3$ which implies influence
(again easy to work through or use Proposition~\ref{prop:influence} below).
Therefore the space of 
probabilities satisfying both influence and independence for $n = 3$
consists of two connected pieces formed by removing the line $p_1 =
p_2 = p_3$ from the plane $p_1-4p_2+3p_3 = 0$ (the fact that this
space is disconnected also follows from Theorem~\ref{gstConnected} and
Corollary~\ref{cor:connected}).

\subsection{Characterizing influence}
For the fully symmetric model we can characterize when the system has
influence. First, however, a brief discussion of influence is
useful. We say a particular cause $C_j$ influences a particular effect
$E_i$ if there exists an assignment of states for the remaining $n -1$
causes such that changing the state of $C_j$ changes the probability
of at least one state of $E_i$.  
We might consider two types of influence:
\begin{itemize}
\item[($I_1$)] For every $1\leq i\leq n$,  the cause $C_i$ influences the
  effect $E_i$.  
\item[($I_2$)] For every $1\leq i,j\leq n$, the cause $C_i$ influences the
  effect $E_j$.  
\end{itemize}

The statement ($I_2$) matches Definition~\ref{def:influence} and is stronger than ($I_1$).
However, in thinking about applications, like the flowers blooming
early versus late, we are largely concerned with the flower's cause
influencing its own effect which is likely to be subject to natural 
selection.  In the symmetric case, these two types of influence are
equivalent, which  we establish in the next proposition, along with a
characterization of influence in terms of the probabilities $p_i$.  

\begin{prop}\label{prop:influence}
Assume the GST setting, so $r = 1/2$ and $p_i = q_i$.  Then
the following are equivalent:
\begin{itemize}
\item[(i)] The system satisfies {\rm ($I_1$)};
\item[(ii)] The system satisfies {\rm ($I_2$)};
\item[(iii)] There exists $s\in[n]$ such that $p_s\neq p_{n-s+1}$. 
\end{itemize}
\end{prop}

\begin{proof}
We argue that (i) $\Rightarrow$ (iii) $\Rightarrow$ (ii) $\Rightarrow$
(i).  If a system satisfies ($I_2$), it obviously
satisfies ($I_1$), so (ii) $\Rightarrow$ (i). 

((i) $\Rightarrow$ (iii))  We prove the contrapositive.  Assume that $p_s =
  p_{n-s+1}$ for all $1\leq s\leq n$.
  Then 
$$\PP(E_i \mid C_i = 0 \bigwedge_{j\neq i} C_j = x_j) = p_{k+1} =
  p_{n-k} = \PP(E_i\mid C_i = 1 \bigwedge_{j\neq i} C_j = x_j),$$
  where $k$ is the number of zeros occurring in the sequence 
$(x_j:  j \neq i)$.  Therefore $C_i$ has no influence on $E_i$ and the
  system fails ($I_1$).     

((iii) $\Rightarrow$ (ii))  Suppose that $p_s\neq p_{n-(s+1)}$ for some
$s\in [n]$.  As above, since $$\PP(E_i \mid C_i = 0
\bigwedge_{j\neq i} C_j = x_j) = p_{k+1}\neq p_{n-k} = \PP(E_i\mid C_i
= 1 \bigwedge_{j\neq i} C_j = x_j),$$ where 
  $k$ is the number of zeros occurring in the sequence $(x_j:  j \neq
i)$,  $C_i$ influences $E_i$.  Observe that if $p_s\neq p_{n-(s+1)}$ for some
$s\in [n]$, there must exist some $t\in[n]$ such that $p_t\neq
p_{t+1}$.  Let $j\neq i\in [n]$.  Set $x_k = 0$ for any $t-1$
  values of $k\neq i,j$.  Then 
$$\PP(E_i \mid C_i = 0, C_j = 0, \bigwedge_{k\neq i,j} C_j = x_j) = p_{t+1} \neq
  p_{t} = \PP(E_i \mid C_i = 0, C_j = 1, \bigwedge_{k\neq i,j} C_j =
  0).$$  Therefore each $C_j$ influences each $E_i$ for all $i,j\in
  [n]$ and the system satisfies ($I_2$).  

\end{proof}

To aid in further discussions, set $\Inf_n$ to be the set of points 
${\bf p}\in [0,1]^n$ such that the system has influence.  

\subsection{Characterizing independence} 
We continue to assume the GST setting, that is $r = 1/2$ and $p_i = q_i$.  
For the vector ${\bf p} = (p_1,p_2, \ldots, p_n)$, let 
\begin{equation}\label{eq:gstIndependence}
\psi({\bf p})=\bigg(\frac{1}{2^{n-1}}\sum_{k = 0}^{n-1}\binom{n-1}{k}p_{k+1}\bigg)^2 - \frac{1}{2^{n-1}}\sum_{k = 0}^{n-2}\binom{n-2}{k}(p_{k+2}^2 + p_{k+1}p_{n-(k+1)}).
\end{equation}

The function $\psi$ allows us to characterize independence as follows.
\begin{prop}
\label{psiprop}
The effects are pairwise independent (equivalently, each pair of effects is screened off by any cause) if and only if
$\psi({\bf p}) = 0.$
\end{prop}

\begin{proof}
The symmetry in the GST model implies that for all $i,j \in
\{1,\ldots, n\}$ 
$$\PP(E_i)= \PP(E_i \mid C_j= x) = \PP(E_1 \mid C_1 = 0).$$ 
This last  probability can be expressed as the sum
over $k=0, \ldots n-1$ of the binomial probability
($\binom{n-1}{k}2^{-(n-1)}$)  that $k$ of the causes $C_2, \ldots,
C_n$ are also in state 0, times the probability ($p_{k+1}$) of $E_1$
given this event and given that $C_1=0$.  This leads to: 
$$\PP(E_i)\PP(E_j) = \PP(E_1|C_1=0)^2 = \bigg(\frac{1}{2^{n-1}}\sum_{k = 0}^{n-1}\binom{n-1}{k}p_{k+1}\bigg)^2.$$
Similarly, for any $i \neq j$ 
$$\PP(E_i \wedge E_j) = \PP(E_1\wedge  E_2|C_1=0).$$
We consider two cases here: either $C_2 =0$ or $C_2=1$, each of which
has probability 1/2. In the first case,  $\PP(E_1 \wedge E_2|C_1=0, C_2=0)$
can be expressed as the sum over all $k=0, \ldots n-2$ of the binomial
probability ($\binom{n-2}{k}2^{-(n-2)}$)  that $k$ of the causes $C_2,
\ldots, C_n$ are also in state 0, times the probability ($p_{k+2}^2$)
of $E_1$ and $E_2$ given this event and given that   $C_1=0$ and
$C_2=0$.  This leads to the first term on the right-hand side of the
expression for $\psi({\bf p})$. An analogous argument for the case
where $C_2=1$ leads to the second term on the right. Notice that the
factor $1/2$ $(=\PP(C_2=0)=\PP(C_2=1)$) gives the required power of $2$
as: $\frac{1}{2} \times 2^{-(n-2)} = 2^{-(n-1)}$. 

\end{proof}

Again, to aid our discussion set 
$$\Ind_n := \{{\bf p}\in [0,1]^n\mid\psi({\bf p}) = 0\};$$ 
that is $\Ind_n$ is the set of all points so
 that the system has independence.  Finally, we set 
$$\GST_n := \Ind_n\cap\Inf_n.$$  While our discussion is entirely in the GST
setting, when talking about subsets of $[0,1]^n$ we will only use
$\GST_n$ when the system has both influence and independence.

\section{Some special points in $\GST_n$}\label{sec:points}

Before we dig deep into the geometric and topological structure of $\GST_n$, we show the
space is non-empty by explicitly establishing a few useful points in
the space, starting with $\Ind_n$ and moving on to points that are in
$\GST_n$.  

The quadratic form discussed in the next
section gives us an easy way, from details in the proof of
Theorem~\ref{gstConnected},  to
show  that there are infinitely many points in
$\GST_n$.  However, we found the following explicit
points useful for proving that both $\GST_n$ and $\Ind_n$ are not
convex. These examples also 
illustrate the challenge of trying to write down explicit points.

\subsection{Explicit points in $\Ind_n$}
If $p_i = p$ for all $i \in \{1,\ldots, n\}$ then any two events $E_i$ 
and $E_j$, where $i\neq j$,  are independent (equivalently,  they are screened
off by a single cause $C_k$ for any $ k 
\in \{1,\ldots, n\}$), i.e.  $\psi(p,p,\ldots, p) = 0$. It is
relatively easy to establish this fact explicitly, but it also follows
directly from the fact that $\psi$ is a quadratic form and $(1,\ldots,
1)$ is an eigenvector for its Hessian matrix with eigenvalue $0$ (see
Section~\ref{sec:quadForm} and Proposition~\ref{eigenvalues}).  This point
fails influence by Proposition~\ref{prop:influence}.

Furthermore, when $n$ is odd, easy computations show that
$p_i=p$ (for $i$ odd) and $p_i = p'$ (for $i$ even), where 
$0<p, p' <1$ satisfies independence.  However, this also fails the
influence requirement, since when $n$ is odd, $i$ is odd/even if and
only if $n-i+1$ is odd/even and therefore
Proposition~\ref{prop:influence} implies no influence.  
 
An alternative approach to try to achieve independence and influence
simultaneously using two parameters $p \neq p'$ is to select ${\bf p}$
so that influence applies, and then attempt to enforce independence. For
example, if we select some $j \in \{1,\ldots, n\}$ where $j \neq n/2$
and define ${\bf p}$  by setting: 
 $$p_i = \begin{cases}
 p, & \mbox{ if } i \neq n-j;\\
 p', & \mbox{ if } i = n-j;
 \end{cases}
 $$
then it is clear from  Proposition~\ref{prop:influence} that influence
holds.   However, it is easy to show that independence fails in this
case, illustrating how challenging it can be to  write down
points in $\GST_n$ explicitly.  

\subsection{Explicit points in $\GST_n$ with all coordinates non-zero}
\label{explicit2}
In the next two subsections we explicitly compute points in $\GST_n$
which are particularly useful for showing that $\Ind_n$ is not convex. 
For the first set of points set $p_k = \theta^k$ for some 
$0 < \theta < 1$.  Then $p_i\neq p_j$ for all $i\neq j$, which   
implies influence.  We claim there exists at
least one $\theta$ that implies independence of effects. 
Since we are in the GST setting we use Eqn.~(\ref{eq:gstIndependence})
and substitute $\theta^k$ for $p_k$ to obtain:
\begin{align}\label{eq:theta1}
\psi({\bf p})
= & \bigg(\frac{1}{2^{n-1}}\sum_{k=0}^{n-1}\binom{n-1}{k}\theta^{k+1}\bigg)^2
- \frac{1}{2^{n-1}}\sum_{k = 0}^{n-2}\binom{n-2}{k}((\theta^{k+2})^2 +
\theta^{k+1}\theta^{n-(k+1)})\notag\\
= & \bigg(\frac{1}{2^{n-1}}\theta (1+\theta)^{n-1}\bigg)^2
- \frac{1}{2^{n-1}}(\theta^4(1+\theta^2)^{n-2} + 2^{n-2}\theta^n)\notag\\
= &\frac{1}{2^{2n-2}}\theta^2\bigg((1+\theta)^{2n-2} -
2^{n-1}\theta^2(1+\theta^2)^{n-2} - 2^{2n-3}\theta^{n-2}\bigg).
\end{align}
To determine $\theta$ such that one cause screens off two events we need to
determine when Eqn.~(\ref{eq:theta1}) is equal to zero.  
Of course, $\theta = 0$ is a solution but it fails to satisfy influence, by
Proposition~\ref{prop:influence}.  So we study the equation  
\begin{equation}\label{eq:theta}
(1+\theta)^{2n-2} - 2^{n-1}\theta^2(1+\theta^2)^{n-2} -
  2^{2n-3}\theta^{n-2} = 0.
\end{equation}

For $n = 3$, we can explore the structure in two ways, using either 
Eqn.~(\ref{nIs3}) or Eqn.~(\ref{eq:theta}).  In the case of Eqn.~(\ref{nIs3}),
the polynomial factorizes as 
$$\theta^2(1-\theta^2)(1-4\theta+3\theta^2)  = 0,$$ 
and without the $\theta^2$ in the case of Eqn.~(\ref{eq:theta}).  The
solutions $\theta = 0$ and $\theta = 1$ correspond to no influence
by Proposition~\ref{prop:influence}, and $\theta = -1$ is not
stochastic. That leaves $1-4\theta+3\theta^2 = 0$ which factorizes as
$(1-3\theta)(1-\theta)$, showing two solutions: $\theta = 1$ and $\theta
= 1/3$.  Therefore, for $n = 3$, there is one value of $\theta$ which
is stochastic and all the probabilities involved are distinct, so 
the causes influence the effects (i.e.  the system satisfies influence). Note
that $\theta = 1/3$ provides a different point in $\GST_n$ than that used
in~\cite{SS}.  

Set $f(\theta) = (1+\theta)^{2n-2} - 2^{n-1}\theta^2(1+\theta^2)^{n-2} -
  2^{2n-3}\theta^{n-2}$.  Notice that 
\begin{align*}
f(0) & = 1,\\
f(1) & = 2^{2n-2} - 2^{2n-3}- 2^{2n-3} = 0. 
\end{align*}
We use these facts, the behavior of $f(1/n)$ as $n$ tends to
infinity, and the Intermediate Value Theorem, to study the zeros of
$f(\theta)$.  We proceed with a study of $f(1/n)$ as $n$ tends to
infinity.  
\begin{align*}
f\bigg(\frac{1}{n}\bigg) = & 
 \bigg(1+\frac{1}{n}\bigg)^{2n-2} -
 2^{n-1}\frac{1}{n^2}\bigg(1+\frac{1}{n^2}\bigg)^{n-2} -
 2^{2n-3}\frac{1}{n^{n-2}}  
\end{align*}
The first term tends to $e^2$ as $n$ becomes large.  The last term tends
to $0$ as $n$ tends to infinity. The
middle term tends to $-\infty$ since $2^{n-1}\gg \frac{1}{n^2}$ and
$\bigg(1+\frac{1}{n^2}\bigg)^{n-2}$ tends to 1.  Thus for large $n$,
$f(1/n) < 0$.  Since $f(0) = 1$, the Intermediate
Value Theorem establishes that $f$ has a root between $0$ and $1/n$
for all large $n$.  We note that we can determine numerically that
$f(1/n) >0 $ for $n \leq 10$ and $f(1/n)<0$ for all $n \geq 11$.

A few graphs of $f(\theta)$, in Fig.~\ref{fig:ftheta},  are instructive.  We use a window that
makes the roots easy to observe on the interval $[0,1]$, but this cuts
off some of the extreme parts of the curves as $n$ increases. First we
notice that a root between $0$ and $1/n$ appears in the graph of
$f(\theta)$ for $n = 10$, but the argument above only guarantees it
for $n \geq 11$ and it does not appear in the graph for $n = 9$ (or
smaller).  However, there is another root (and, once $n \geq 10$, two
other roots) that appear to be converging to $1$ rather than $0$ and
this root already appears  for $n\geq 3$.  

\begin{figure}
        \centering
        \begin{subfigure}[b]{0.3\textwidth}
                \centering
                \includegraphics[width=\textwidth]{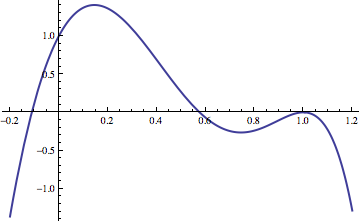}
                \caption{$n = 4$}
                \label{fig:f4}
        \end{subfigure}%
        ~\qquad 
        \begin{subfigure}[b]{0.3\textwidth}
                \centering
                \includegraphics[width=\textwidth]{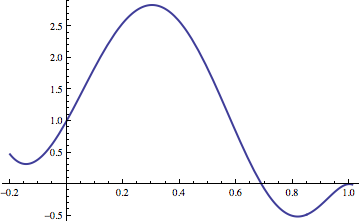}
                \caption{$n = 5$}
                \label{fig:f5}
        \end{subfigure}
        ~\qquad 
        \begin{subfigure}[b]{0.3\textwidth}
                \centering
                \includegraphics[width=\textwidth]{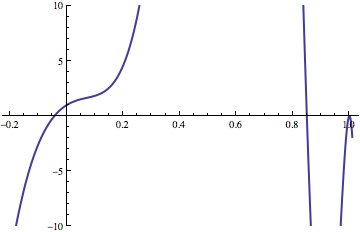}
                \caption{$n=9$}
                \label{fig:f9}
        \end{subfigure}
\qquad 
\begin{subfigure}[b]{0.3\textwidth}
                \centering
                \includegraphics[width=\textwidth]{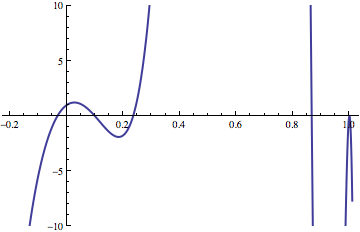}
                \caption{$n=10$}
                \label{fig:f10}
        \end{subfigure}
~\qquad 
\begin{subfigure}[b]{0.3\textwidth}
                \centering
                \includegraphics[width=\textwidth]{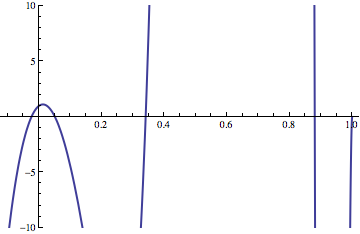}
                \caption{$n=11$}
                \label{fig:f11}
        \end{subfigure}
        \caption{Graphs of $f(\theta)$.}\label{fig:ftheta}
\end{figure}

\vspace{0.2cm}

In summary, there exists points in $\GST_n$ for any $n\geq 11$.  For $n\leq 10$ we
can explore the system numerically to determine that $f$
still has a root in $(0,1)$. We use these points in our discussion of
convexity in Section~\ref{sec:geometry}.

\subsection{Explicit points in $\GST_n$ with many zero coordinates}\label{bdryPoints}
A second way to construct explicit elements of $\GST_n$ is
to look at `boundary points'.  
\begin{prop}\label{canonicalPoint}
For any $n\geq 4$, there is exactly one value of $p_n$ such that the point ${\bf p} =
(1, 0, 0, 0, 0,\ldots, 0, p_n)$ lies in $\GST_n$.  
\end{prop}

\begin{proof}
To simplify initial computations, we let $N = 2^{n-1}$ to obtain:
\begin{align*}
\psi({\bf p}) = & \frac{1}{N^2}\bigg(\sum_{k=0}^{n-1}\binom{n-1}{k}p_{k+1}\bigg)^2 -
\frac{1}{N}\sum_{k=0}^{n-2}\binom{n-2}{k}p_{k+2}^2 -
\frac{1}{N}\sum_{k=0}^{n-2}\binom{n-2}{k}p_{k+1}p_{n-(k+1)} \\
= & \frac{1}{N^2}(1 + p_n)^2 - \frac{1}{N}(p_n^2). 
\end{align*}
Thus the quadratic formula gives 
\begin{equation*}
p_n = \frac{-2 \pm \sqrt{4 - 4(1-N)}}{2(1-N)} = \frac{-1 \pm
  \sqrt{N}}{1-N}.
\end{equation*}
Then for any $N>1$, one root lies between $0$ and $1$, namely
$\frac{-1-\sqrt{N}}{1-N} = \frac{1}{\sqrt{N}-1}$.   
The point ${\bf p} = (1, 0, 0, 0, 0,\ldots, 0,
\frac{1}{\sqrt{N}-1})$ is in  $\GST_n$, since it also satisfies influence as $1\neq
\frac{1}{\sqrt{2^{n-1}}-1}$ for any $n\geq 4$.  
\end{proof}
The computations in the proof above work for $n = 3$,
but when $n = 3$,  $\frac{1}{\sqrt{2^{3-1}}-1} = 1$.  Therefore, the point we
get, using this approach is $(1,0,1)$, which satisfies independence, 
but not influence.  Similar computations (or Remark~\ref{involution} below)
show that 
${\bf 1}-{\bf p} = (0, 1, \ldots, 1, \frac{\sqrt{N}-2}{\sqrt{N}-1})\in\GST_n$ 
as well.

\section{The quadratic form $\psi$}\label{sec:quadForm}

To understand $\GST_n$, we find it helpful to study the structure of
$\psi$ as given in Eqn.~(\ref{eq:gstIndependence}).  For example, 
the first partial derivatives of $\psi$ are zero at ${\bf p} = (1/2,
1/2, \ldots, 1/2)$.  This turns out to be one piece of evidence that
this point is special (another is that there are lots of
lines, which are mostly in $\GST_n$, passing through this point, as we
show later).  However, since $\psi$ is a quadratic form, the Hessian
matrix, denoted $H_n$,  seems to be more
helpful in our study of $\GST_n$ near  the point ${\bf p} = (1/2,
1/2, \ldots, 1/2)$ and more generally.  

To compute the
Hessian matrix we begin with the first derivative.  Throughout this
section we use $N = 2^{n-1}$ to simplify expressions. For all $i \neq
1, n$, 

\begin{equation}\label{Gradient}
\frac{\partial\psi}{\partial p_i} =
\frac{2}{N^2}\binom{n-1}{i-1}\bigg(\sum_{k=0}^{n-1}\binom{n-1}{k}p_{k+1}\bigg)
- \frac{2}{N}\bigg[\binom{n-2}{i-2}p_i +
  \binom{n-2}{i-1}p_{n-i}\bigg].  
\end{equation}
  When $i = 1$ simply remove the term
  $\frac{2}{N}[\binom{n-2}{i-2}p_i]$ and when $i = n$ remove the term
  $\frac{2}{N}[\binom{n-2}{i-1}p_{n-i}]$.  
From this the second partials are easy to compute.
\begin{equation}\label{Hessian}
\frac{\partial^2\psi}{\partial p_i\partial p_j} = \frac{2}{N^2}\binom{n-1}{i-1}\binom{n-1}{j-1} - \begin{cases}
\frac{2}{N}\binom{n-2}{i-2} & i = j \neq 1,\frac{n}{2};\\
\frac{2}{N}\binom{n-2}{i-1} & j = n-i,j\neq \frac{n}{2};\\
\frac{2}{N}\binom{n-2}{i-2} + \frac{2}{N}\binom{n-2}{i-1} & i = j = \frac{n}{2};\\
0 & {\text otherwise}.
\end{cases}
\end{equation}
Since $\psi$ is a quadratic
polynomial, the Hessian matrix is constant, as expected.  Furthermore, since $\psi$ is a quadratic form
corresponding to a symmetric matrix we label $Q_n$,  $H_n= Q_n +
Q_n^T = 2Q_n$.  Therefore, knowing $H_n$ gives us $Q_n$ as well. 

To determine for which values of $n$ the space $\GST_n$ is connected --  our main
goal -- we need several results regarding the eigenvalues and
eigenspaces of the Hessian matrix $H_n$, which we collect here.   

 \begin{prop}
\label{eigenvalues}
For all $n\geq 3$, the Hessian matrix $H_n$ has $0$ as an eigenvalue with associated eigenvector ${\bf 1}$.  
\end{prop}

\begin{proof}
The vector ${\bf 1}$ is an eigenvector for the eigenvalue $0$ if and only if the row sums are
$0$.    The sum of the the entries in the $i^{th}$ row of
$H_n$, for $i\neq 1,n$ (it does not matter here whether $n$ is even or odd), using
Eqn.~(\ref{Hessian}),  is 
\begin{equation*}
\sum_{j = 1}^n \frac{2}{N^2}\binom{n-1}{i-1}\binom{n-1}{j-1} -
\frac{2}{N}\binom{n-2}{i-2} - \frac{2}{N}\binom{n-2}{i-1} =
\frac{2}{N}\binom{n-1}{i-1}- \frac{2}{N}\binom{n-1}{i-1} = 0.
\end{equation*}
This uses  $\sum_{j = 1}^{n}\binom{n-1}{j-1} = 2^{n-1} = N$ and
$\binom{n-2}{i-2} + \binom{n-2}{i-1} = \binom{n-1}{i-1}$.  The
arguments for $ i = 1, n$ are similar, with simpler computations.  
\end{proof}

\begin{remark}\label{Hstructure}
Observe from Eqn.~(\ref{Hessian}) that the Hessian matrix $H_n =
{\bf v}{\bf v}^T - X$, where ${\bf v}$ is the vector with $i^{th}$ entry equal to 
$\frac{\sqrt{2}}{N}\binom{n-1}{i-1}$.  The matrix $X$ has non-zero
entries on the diagonal, except for the $(1,1)$ location, which is
$0$, and there are non-zero entries on the
opposite diagonal given by $i + j = n$. 
For example, below are the matrices $X$ for $n =
4$ and $n = 5$,  in both cases scaled by multiplying by $N/2 = 2^{n-2}$. 
These two cases also illustrate the differences in $X$ for odd vs. even values of $n$.  Finally, it
is helpful to keep the shape of this matrix in mind for many of the
following arguments.   
\end{remark}

\begin{equation*}
\setlength{\arraycolsep}{1pt}
\begin{bmatrix}
0 						&         			&	0    					&					& \binom{2}{0} 	& 				& 0 \\[-5pt]
							& \ddots				&								&	\idots	&								& 				& \\[-5pt]
0 						&								& \binom{2}{0} 	& 				& 0  						& 				& 0 \\[-5pt]
							& \idots				&								&	\ddots	&								& 				& \\[-5pt]
\binom{2}{2} 	&				 				& 0 						& 				& \binom{2}{1} 	& 				& 0  \\[-5pt]
							&								&								&					&								& \ddots	& \\[-5pt]
0 						&								& 0 						&					& 0 						&					& \binom{2}{2}
\end{bmatrix}, \quad
\begin{bmatrix}
0 						& 							& 0    					&																	& 0 						&								& \binom{3}{0} 	& 				& 0 \\[-5pt]
							& \ddots				&								&																	&								&		\idots			& 							& 				& \\[-5pt]
0 						& 							& \binom{3}{0} 	& 																&  \binom{3}{1} &								& 0  						& 				& 0\\[-5pt]
							& 							&								&	\text{$\idots$\llap{$\ddots$}}	&								& 							& 							&					& \\[-5pt]
0 						& 							& \binom{3}{1} 	& 																&  \binom{3}{1} & 							&	0 						& 				& 0 \\[-5pt]
							& \idots				&								&																	&								&		\ddots			& 							& 				& \\[-5pt]
\binom{3}{3} 	& 							& 0							& 																& 0 						& 							&\binom{3}{2} 	& 				& 0 \\[-5pt]
							&  							&								&																	&								&								& 							& \ddots 	&	\\[-5pt]
0  						& 							& 0 						& 																& 0 						& 							& 0							& 			 	& \binom{3}{3}
\end{bmatrix}
\end{equation*}

\begin{lem}
The matrix $X$ has rank $n$.  
\end{lem}

\begin{proof}
If $i\neq 1, n-1, n$ or, when $n$ is even, $\frac{n}{2}$, then rows $i$
and $n-i$ each have two entries in the same columns, which are $i$ and
$n-i$. To simplify the discussion, assume, without loss of generality,
that $i< \frac{n}{2}< n-i$, and we indicate row $j$ in the matrix by
$R_j$.  The entries in $R_i$ are 
$\frac{2}{N}\binom{n-2}{i-2}$ and 
$\frac{2}{N}\binom{n-2}{i-1}$, respectively, and in $R_{n-i}$ they are
$\frac{2}{N}\binom{n-2}{n-i-1} = \frac{2}{N}\binom{n-2}{i-1}$ and
$\frac{2}{N}\binom{n-2}{n-i-2}$ respectively.  The standard row
operation replacing $R_{n-i}$ with $-\frac{n-i}{i-1}*R_i +
R_{n-i}$ places a $0$ in the $i^{\mathrm{th}}$ entry in $R_{n-i}$ and
$$\frac{2}{N}\binom{n-2}{i-2}\bigg(\frac{-n+1}{(n-i-1)(i-1)}\bigg)\neq
0$$ in the $n-i^{\mathrm{th}}$ entry.  
Then swap rows $1$ and $n-1$, and observe that row $n$, and, when $n$
is even, row $\frac{n}{2}$ have only one non-zero entry in column $n$, respectively column $\frac{n}{2}$. 
Therefore
$X$ is row-equivalent to an upper triangular matrix where all the
diagonal entries are non-zero. 
\end{proof}

\begin{prop}
\label{eigenspace}
For all $n\geq 4$, the eigenspace corresponding to the eigenvalue $0$ has dimension 1.  
\end{prop}

\begin{proof}
It is enough to prove that $\rank(H_n) = n-1$.    Since $H_n = {\bf
  v}{\bf v}^T - X$, the subaddativity of matrix rank applied to $-X =
H_n - {\bf v}{\bf v}^T$ gives $\rank(X) \leq \rank(H_n) + \rank({\bf
  v}{\bf v}^T)$. Since $\rank({\bf v}{\bf v}^T) = 1$ and $\rank(X) = n$,  
  $n-1\leq \rank(H_n)$.
Since $0$ is an eigenvector, $n-1 = \rank(H_n)$. 
\end{proof}

\smallskip

\begin{remark}~\label{diagonalize}
Since $\psi({\bf x}) = {\bf x}^TQ_n{\bf
  x}$ is a quadratic form, we can diagonalize $Q_n$ using an orthogonal
matrix $P$, that is $P^TQ_nP = D$, where $D$ is a diagonal matrix of
real eigenvalues of $Q_n$.  Since $H_n = 2Q_n$, we could equivalently
write $\psi({\bf x}) = \frac{1}{2}{\bf x}^TH_n{\bf x}$ and diagonalize
$H_n$ instead.  Furthermore, all the results in this section apply
equally to $Q_n$, but are easier to prove and think about in terms of
$H_n$.  However, in later arguments, we use $Q_n$ instead of $H_n$ to
avoid having to keep track of the factor $\frac{1}{2}$.
\end{remark}

We prove in Theorem~\ref{gstConnected} that the connectedness of $\GST_n$
depends on the number of strictly positive and strictly negative
eigenvalues of $H_n$.  We establish here that $H_n$ has ``enough'' of
each type of eigenvalue for $n\geq 6$.  For ease of notation, we use
$H = H_n$ in the following discussion.  

\begin{thm}\label{eigenvalueNumbers}
For all $n\geq 6$, $H$ (equivalently, $Q_n$) has at least two strictly
positive and at least two strictly negative eigenvalues.  
\end{thm}

\begin{proof}
Let $A = H + \epsilon B$ where $\epsilon >0$ and 
$$B_{ij} = \begin{cases} 1, &
  \mbox{if } i+j = n+1;\\ 
0, & \mbox{otherwise.}\end{cases}.$$  
Let $A_k$ denote the submatrix of $A$ consisting of the first $k$ rows
and columns of $A$ so that $\det(A_k)$ is the $k^{\rm th}$ leading
principal minor of $A$. Then $A_k = H_k$ for all $1\leq
k\leq\lfloor \frac{n}{2}\rfloor$. Therefore, for all $1\leq
k\leq\lfloor \frac{n}{2}\rfloor$, $A_k = ({\bf v}{\bf v}^T)_k - X_k$ for a
vector ${\bf v}$ and a matrix $X$, where $X_k$ is diagonal and its first 
entry is $0$ (see Remark~\ref{Hstructure}).  Hence elementary row operations on $A_k$ transform it
into an upper triangular matrix $T$ such that $T_{11} = A_{11} = \frac{2}{N^2}$
and $T_{ii} = X_{ii} = \frac{2}{N}\binom{n-2}{i-2}\neq 0$ for all $2\leq i\leq k$.  Thus $\det(A_k)
\neq 0$ for all $1\leq k\leq \lfloor \frac{n}{2}\rfloor$.  

If $k \geq \lfloor \frac{n}{2}\rfloor + 1$, then $\det(A_k)$ is a
polynomial in $\epsilon$ (for example, when $k = \lfloor
\frac{n}{2}\rfloor+1$, and $n$ is odd, $\epsilon$ appears in the
$(\lfloor \frac{n}{2}\rfloor+1, \lfloor \frac{n}{2}\rfloor+1)$ entry).  Set
$p_k(\epsilon) = \det(A_k)$ for $\lfloor \frac{n}{2}\rfloor + 1 \leq k
\leq n$.  This is a finite set of polynomials, each with a finite
number of zeros.  Call that set of zeros $Z$, and let
\begin{equation}
\label{epsZ}
\epsilon_Z = \min(\{|z| : z\in Z\} - \{0\}),
\end{equation}
which is strictly positive (since $Z$ is finite). 
Then for any $\epsilon \in (0, \epsilon_Z)$ we have that $ \det(A_k)\neq 0$ for all 
$\lfloor\frac{n}{2}\rfloor + 1 \leq k \leq n$. Therefore all the leading principal minors
of $A$ are non-zero (including $\det(A) = \det(A_n)$). 

Since all of the leading principal minors of $A$ are
non-zero, $A$ has a unique $LU$-decomposition \cite[Theorem
  2.13]{PO}.  Since $A$ is symmetric, the $LU$-decomposition can be
transformed into an $LDL^T$-decomposition where $L$ is lower
triangular and $D$ is diagonal~\cite[Theorem 2.14 and discussion]{PO}.
Furthermore, simply writing this
expression out gives the following recursive formulae for the entries
of $D$ and $L$, assuming $i > j$:    
\begin{align}
D_j = &  A_{jj} - \sum_{k = 1}^{j-1}L_{jk}^2D_k \label{D}\\
L_{ij} = & \frac{1}{D_j}\bigg(A_{ij} - \sum_{k=1}^{j-1}
L_{ik}L_{jk}D_k\bigg). \label{L}
\end{align}
We show that $D_1 > 0$, $D_i<0 $ for $2\leq i\leq \lfloor
\frac{n}{2}\rfloor$ and $D_{\lfloor \frac{n}{2}\rfloor + 1} >0$.
Therefore $D$ has at least two strictly negative eigenvalues and two
strictly positive
eigenvalues for $n\geq 6$.  By Sylvester's Theorem~\cite{Sy}, $A$
and $D$ have the same index (or inertia) and hence $A$ also has at least
two strictly negative eigenvalues and two strictly positive eigenvalues for $n\geq
6$. Before digging into computing $D_i$ we argue that $H$ must also
have at least two strictly negative eigenvalues and two strictly positive eigenvalues
for $n\geq 6$.

Over the
complex numbers, roots of a polynomial are 
continuous functions of the coefficients of the
polynomial~\cite[Theorem (1,4)]{Ma} which implies that each
eigenvalue of $A$ corresponds to an eigenvalue of $H$. More formally,
let $p_A(x) = x^n + c_1 x^{n-1} + \cdots + c_n$ denote the 
characteristic polynomial of $A$ and $p_{H}(x) = x^n + d_1 x^{n-1} +
\cdots + d_n $ be the characteristic polynomial of $H$. By
construction, $d_i = c_i + \epsilon_i$ for $1\leq i\leq n$ and
each $\epsilon_i$ approaches $0$ as $\epsilon$ (in the definition of $A$) goes to $0$.  Suppose that:
$$p_A(x) = \Pi_{k=1}^q(x-a_i)^{m_i}$$ with the distinct $a_i\in \RR$, since $A$ is
symmetric. Then for any 
$$0<r_k<\min\{|a_k - a_i|,\ i = 1, 2, \cdots, k-1, k+1, \cdots q\},$$
there exists a $\delta$ such that if $|c_j - d_j|<\delta$ for all
$1\leq j\leq n$, then $p_H(x)$ has $m_k$ roots in a circle of radius
$r_k$ centered at $a_k$.  Since $H$ is also symmetric, its roots are
also real and if $a_k$ is positive (resp. negative), then for small
enough values of  $r_k$, the corresponding roots of $p_H(x)$ are also positive
(resp. negative).  Let $\epsilon$  (in the definition of $A$), be less
than $\epsilon_Z$ from (\ref{epsZ}), and also small enough so that if
$A$ has at least two strictly positive eigenvalues and at least two
strictly negative eigenvalues for $n\geq 6$, then $H$ does also.  

We finish by showing that  $D_1 > 0$, $D_i<0 $ for $2\leq i\leq \lfloor
\frac{n}{2}\rfloor$ and $D_{\lfloor \frac{n}{2}\rfloor + 1} >0$ for
$A$.  Throughout this discussion, we assume $i > j$ and use
Eqns.~(\ref{D}) and (\ref{L}).  For all $i\neq
n-j+1$, $A_{ij} = H_{ij}$.  
Thus $D_1 = H_{11} = \frac{2}{N^2} >0$. Furthermore, 
$$L_{i1} =\frac{1}{D_1}\bigg(D_1\binom{n-1}{i-1}\binom{n-1}{0}\bigg) =
\binom{n-1}{i-1}\mbox{, for } 1\leq i\leq n-1.$$  
Therefore 
\begin{equation*}
A_{ij} = H_{ij} = D_1L_{i1}L_{j1}, \mbox{ for all } i\neq n-j,\ n-j+1.
\end{equation*} 
We use this fact repeatedly throughout the remaining discussion.  
Also note that  $i\neq n-j,\ n-j+1$ for all $1\leq i,j\leq
\lfloor \frac{n}{2}\rfloor$.    Hence, 
$$L_{ij} = -\frac{1}{D_j}\bigg(\sum_{k=2}^{j-1}
L_{ik}L_{jk}D_k\bigg) \mbox{ for all } i\neq n-j,\ n-j+1. $$
Then, by induction on $j$, $L_{ij} = 0$ for all $1 < i,j \leq \lfloor
\frac{n}{2}\rfloor$ since $L_{i2}$ is trivially zero and then the sum
for $L_{ij}$ only includes expressions where the second index is strictly less
than $j$.   Therefore 
\begin{equation*}
D_i =  H_{ii} - \sum_{k = 1}^{j-1}L_{jk}^2D_k =  -
\frac{2}{N}\binom{n-2}{i-2} < 0, \mbox{ for all } 1 < i\leq\bigg\lfloor\frac{n}{2}\bigg\rfloor.
\end{equation*}
Thus  we have $D_1>0$ and, for $n\geq 6$,  at least two strictly negative
eigenvalues for $D$.  

Finally, we need to argue that $D_{\lfloor \frac{n}{2}\rfloor + 1}
>0$.  While the arguments are similar, they differ slightly for 
even and odd values of $n$ and are somewhat technical so we placed them in the appendix.  
When $n\geq 6$ is odd, we get 
$$D_{\lfloor \frac{n}{2}\rfloor + 1} =
\frac{2}{N}\binom{n-2}{\lfloor\frac{n}{2}\rfloor}\bigg(\frac{2}{\lfloor\frac{n}{2}\rfloor-1}\bigg)+\epsilon
> 0,
$$
and when $n\geq 6$ is even, $\lfloor\frac{n}{2}\rfloor = \frac{n}{2}$, so
that   
$$
D_{\frac{n}{2}+1} 
= \frac{2}{N}\binom{n-2}{t-1}\bigg(\frac{n-1}{\frac{n}{2}(\frac{n}{2}-2)}\bigg)
 + \frac{\epsilon^2N}{\binom{n-2}{t-2}} >0.
$$
\end{proof}

\section{The Geometry of  $\GST_n$}\label{sec:geometry}

 The space $\GST_n$ is a bounded (but not closed) subspace of $\mathbb{R}^n$.
Fig.~\ref{fig:n3planes} shows that when $n=3$, this space consists of a
pair of two-dimensional components, each of which is convex.  
 
\smallskip

\begin{remark}~\label{involution}
For any $n$, if ${\bf p}\in \GST_n$ then ${\bf 1}
- {\bf p} = (1-p_1, 1-p_2, \ldots, 1-p_n)\in \GST_n$ as may
be verified either algebraically or, more directly, by the symmetry of the
states 0 and 1 in the GST problem.  Thus the map ${\bf p} \mapsto
{\bf 1} - {\bf p}$ is a involution from the solution space to itself;
in  Fig.~\ref{fig:n3planes} 
this maps each connected component onto the other. This involution
also moves every point, since the unique fixed point has $p_i=1/2$ for
all $i$ and this point fails influence.  

Furthermore, if ${\bf p}\in \GST_n$ lies in the GST solution space then for any
constant $0<c \leq 1$, the scaled vector $c\cdot {\bf p}\in \GST_n$,
since $\psi$ is a homogeneous quadratic in the coordinates of ${\bf p}$.  
\end{remark}

These observations are part of the following more general result. 

\begin{prop}
\label{affine}
\mbox{ }

\begin{itemize}
\item[(i)]
For any real values $x$ and $y$ and real vector ${\bf p} =(p_1, \ldots, p_n)$,
$$\psi(x{\bf p} + y{\bf 1}) = x^2\psi({\bf p}).$$ 

\item[(ii)] In particular, if ${\bf p} \in [0,1]^n$ satisfies independence then
$x{\bf p} + y{\bf 1}$ does also, provided this vector also lies in $[0,1]^n$. 
\end{itemize}
\end{prop}

\begin{proof}
Part (i) holds for $y=0$, since $\psi$ is a homogeneous quadric
polynomial, so it suffices to establish part (i) when $x=1$.   In that
case, if we replace  $p_i$ by $p_i+y$ in 
$\psi$, we see that the coefficient of $y^2$ is $\psi(y{\bf 1}) = 0$,
and the coefficient of $y^0$ is $\psi({\bf p})$.  The remaining terms
correspond to the coefficient of $y^1$.  Checking that  this coefficient is equal to 0
requires more careful algebraic analysis (and the use of the combinatorial identity:
$\binom{n-2}{k-1}+\binom{n-2}{k} = \binom{n-1}{k}$), but the
computation is straightforward. This establishes part (i). Part (ii) now follows
from Proposition~\ref{psiprop}. 
\end{proof}

This proposition has a few consequences of note. First, it provides
an alternative argument for the point made in
Remark~\ref{involution}. 
However, it proves further that if ${\bf p}\in
\Ind_n$ then the entire line between ${\bf p}$ and ${\bf 1}-{\bf p}$
also lies in $\Ind_n$.  
Note that any such line must pass through the `middle point' of $[0,1]^n$, namely 
$${\bf m} = (1/2, 1/2, \ldots, 1/2),$$ 
and this point will play an important role in forthcoming arguments.

Furthermore, if we want to explore
points near ${\bf m}\in \Ind_n$ (which is helpful for the proof 
of Theorem~\ref{gstConnected}) -- say, points of the form ${\bf p}
= (1/2 + x_1, \ldots, 1/2 + x_n)$ where $-1/2 < x_i < 1/2$ -- then ${\bf
  p} \in \Ind_n$ if and only if $\psi (x_1, \ldots, x_n) = 0$.  Note
that $(x_1, \ldots, x_n)$ may or may not be in $\Ind_n$ since the
coordinates may or may not all be non-negative.  The question of which
of these points are in $\GST_n$ is a bit more subtle but, generally,
they will be so if  ${\bf p}\in \GST_n$ to start with.  

\begin{remark}\label{equivalence}
Let ${\bf p}, {\bf q}\in \Ind_n$.  We note that
Proposition~\ref{affine} gives an equivalence relation on $\Ind_n$.
We say ${\bf p}\sim {\bf q}$ if and only if ${\bf p} = a{\bf q} +b{\bf
  1}$ for some $a,b\in \mathbb{R}$ with $a\neq 0$. For example, the two
points given in Section~\ref{bdryPoints} are equivalent, as are the
two solutions to $\GST_3$ shown in Fig.~\ref{hands} (use $a = \frac{9}{4}$
and $b = \frac{1}{4}$).  Also note that if ${\bf p},{\bf q}\in [0,1]^n$
and $  {\bf p}\sim {\bf q}$ then ${\bf p}\in\GST_n$ if and only if
${\bf q}\in \GST_n$.  
\end{remark}

The more
general expression $\psi(x{\bf p} + y{\bf q})$ for two points ${\bf p}$ and ${\bf q}$ in $\RR^n$  is
helpful for investigating the convexity of $\Ind_n$ and $\GST_n$, and is 
useful for our next result regarding the equivalence relation $\sim$.  
\begin{align*}
\psi(x{\bf p} + y{\bf q}) = &  
(x{\bf p} + y {\bf q})^T Q_n (x{\bf p} + y {\bf q})\\
= &  x^2\psi({\bf p}) + y^2\psi({\bf q}) + xyCT({\bf p}, {\bf q})
\end{align*}
where the `cross term' CT is given by
\begin{align}\label{eq:crossTerm}
CT({\bf p},{\bf q}) = 2{\bf p}^TQ_n{\bf q}.
\end{align}

It is the \emph{cross term} that we are concerned with in our study of
GST space since the line between two arbitrary points ${\bf p}$ and ${\bf q}$ in $\Ind_n$ lies in $\Ind_n$ if and only
if the cross term $CT({\bf p}, {\bf q})$ is zero.  



\subsection{A geometrically special point in $\Ind_n$}

\begin{prop}\label{allLines}
For any $n \geq 3$,  a point ${\bf x}\in \Ind_n$ has the property that for all
 ${\bf  p} \in \Ind_n$ the line segment from ${\bf p}$ to ${\bf x}$ lies in
  $\Ind_n$ if and only if ${\bf x} \sim {\bf 1}$. 
 
  \end{prop}
  
  \begin{proof}
  
The `if' direction is readily established. 
If ${\bf x} \sim {\bf 1}$ and  ${\bf p} \in \Ind_n$ then Eqn.~(\ref{eq:crossTerm}) and the identity $Q_n {\bf 1} = {\bf 0}$,
implies that  $CT({\bf p}, {\bf x}) =0$.   Thus, $\psi(t{\bf p} + (1-t){\bf x}) = 0$ for all $t \in
[0,1]$, and thus each point on this line lies in $\Ind_n$. 

\bigskip

For the `only if' part,  suppose that ${\bf x} \in [0,1]^n$ satisfies
the property described  (we will say that ${\bf x}$ is {\em
  permissible}).    For all ${\bf q} \in [-1/3, 1/3]^n$ for which $\psi({\bf q})=0$ we have ${\bf m} + {\bf q} \in \Ind_n$ by Proposition~\ref{affine}(ii).   Thus, since ${\bf x} \in \Ind_n$ and by the special assumption concerning
this point, we have:
  $$0=CT({\bf x}, {\bf m} + {\bf q}) =  CT({\bf x}, {\bf m}) + CT({\bf
  x}, {\bf q}) = 0 +CT({\bf x}, {\bf q}),$$ 
which gives
  \begin{equation}
  \label{CTw}
  CT({\bf x}, {\bf q}) =0
  \end{equation}
  for all ${\bf q} \in [-1/3, 1/3]^n$ for which $\psi({\bf q})=0$.
  Let $P$ and $D$ be as given in Remark~\ref{diagonalize}.   If we let (fixed) ${\bf y} = P^T {\bf x}$  
and (variable) ${\bf z} =   P^T {\bf q}$, then for all ${\bf z} \in B = P^T[-1/3, 1/3]^n$ for which
${\bf z}^TD{\bf z}=0$ (i.e. $\psi({\bf q})=0$) we have (from (\ref{CTw})):
\begin{equation}
\label{befeq}
2{\bf y}^TD {\bf z} = 0.
\end{equation}
By Proposition~\ref{eigenspace}, we can order the diagonal entries $D$
as $d_1, \ldots, d_n$  so that $d_1=0$, and $d_j\neq 0$ for $j>1$. 
Set $c_i= d_i y_i$ for  each $i$. 
Then for all ${\bf z}$ in $B$ for which 
\begin{equation}
\label{befeq2}
\sum_{i=2}^n d_i z_i^2 = 0,
\end{equation}   we must also have (from Eqn.~(\ref{befeq})):
$$\sum_{i=2}^n c_i z_i =0.$$

Now, $D$ not only has $n-1$ non-zero eigenvalues, but at least one is
strictly positive and at least one is strictly negative. This is
readily verified for $3 \leq n \leq 5$, and for $n \geq 6$ it is an
immediate consequence of the stronger result stated in
Proposition~\ref{eigenvalueNumbers}.  Consequently, for any $j>1$, the
equation $\sum_{i=2}^n d_i z_i^2 = 0$ has a solution for ${\bf z} \in
B$ with $z_j \neq 0$.    

Now, suppose that $c_j \neq 0$ for some value of $j$.  Let ${\bf z}$
be a vector in $B$ that satisfies Eqn.~(\ref{befeq2}) and has $z_j
\neq 0$,  and let ${\bf z}'$ be the vector obtained from ${\bf z}$ by
flipping the sign of $z_j$ while leaving the $z_i$ values unchanged
for all $i \neq j$. Then ${\bf z}'$ still lies in $B$ and satisfies
Eqn. (\ref{befeq2}) but $\sum_{i=2}^n c_iz_i$ and $\sum_{i=2}^n c_i
z'_i$ cannot both be zero, since they differ by a term of
magnitude $2|c_iz_i| \neq 0$. Thus if ${\bf x}$ is permissible then
$c_i$ must be zero for all $i>1$ and since $d_i \neq 0$ for all $i>1$, we must have: 
$$y_2=y_3=\ldots y_n=0.$$
Thus,  the set of possible values of ${\bf y}$ for which ${\bf x}$ is
permissible is precisely the set  $$\{{\bf y}=(y,0,0,\ldots, 0): P
{\bf y} \in [0,1]^n\},$$  and this is simply $\{p\cdot {\bf 1}: p \in
[0,1]\}$, since $(1,1,\dots, 1)$ is the eigenvector of $H_n$
corresponding to $0$.  

\end{proof}

\subsection{Convexity}

As previously noted, Proposition~\ref{affine} shows that if ${\bf p}\in \GST_n$ then 
$1-{\bf p}$ and the line segment $ (1-t){\bf p} + t(1-{\bf p})$, for
$0\leq t\leq 1$,  between them are all in $\Ind_n$.  Easy computations
show that the point ${\bf m} = (1/2, \ldots, 1/2)$ lies on the line $(1-t){\bf
  p} + t(1-{\bf p})$ for any point ${\bf p}$ but ${\bf m}$
fails influence and hence is not in $\GST_n$. Therefore $\GST_n$ is
not convex. However, in this example, all
the points still lie in independence space and so it might still seem
possible that 
$\Ind_n$ is convex.  Using the cross term given in Eqn.~(\ref{eq:crossTerm}) and the
points from Section~\ref{sec:points}, we show, more strongly, that there are
points in $\GST_n$ where the line between them does not lie in
$\Ind_n$ and hence independence space is not convex either.  

If $n\geq 10$, then we can use two different solutions to $f(\theta) =
0$ to test the convexity of the space by evaluating $CT$.  The
polynomial $f(\theta)$
has three solutions when $n = 10$, two of which are approximately
$.100499$ and $0.86659$.  If we set ${\bf p} = (0.100499, 0.100499^2,
\ldots, 0.100499^{10})$ and ${\bf q} = (0.86659, 0.86659^2, \ldots,
0.86659^{10})$, we have two points in $\GST_n$ such
that $CT({\bf p},{\bf q}) = 30.0527$.  Thus every point on the line $t{\bf p} +
(1-t){\bf q}$ , except for ${\bf p}$ and ${\bf q}$, is outside
independence space and hence outside $\GST_n$.

For smaller values of $n$, we get only one point in $\GST_n$ from
looking at $f(\theta)$, but we can use one of the points $(0,1,\ldots, 1,
\frac{\sqrt{N}-2}{\sqrt{N}-1})$ or $(1,0,\ldots, 0,
\frac{1}{\sqrt{N}-1})$ along with the one point obtained using
$f(\theta)$ to produce points where the line between them lies
entirely outside $\GST_n$.   
 
Let ${\bf p}\in \GST_n$.  Proposition~\ref{affine} shows that
$\frac{1}{\max{\bf p}}{\bf p} \in \Ind_n$ and that every
point on the line segment connecting ${\bf p}$ and $\frac{1}{\max{\bf p}}{\bf p}$ is in
independence space.  Furthermore since each point on the line is a
non-zero multiple of ${\bf p}$, they satisfy influence and hence the
entire line is in $\GST_n$.  

These results show that there are pairs of points for $n\geq 4$ in
$\GST_n$ where (i) the line between them lies entirely outside the space,
(ii) exactly one point lies outside the space and (iii) the line is
entirely inside the space.    

\section{The Topology of $\GST_n$}\label{sec:topology}

As noted previously, the space $\GST_n$ is a bounded (but not closed) subspace of $\mathbb{R}^n$.
Fig.~\ref{fig:n3planes} shows that when $n=3$, this space consists of a
pair of two-dimensional components, each of which is contractable.  

\subsection{Contractable}
Recall that a space is {\em contractable} if it can be continuously
shrunk to a point (i.e. if the identity map is homotopic to the
constant map).  
\begin{prop}
For each $n \geq 3$, $\Ind_n$ is contractable, but $\GST_n$ is not.
\end{prop}
\begin{proof}
For $\Ind_n$,  select any point ${\bf x}\in \Ind_n$ for which ${\bf
  x}\sim {\bf 1}$ (e.g. ${\bf x} = {\bf 0}$ or ${\bf m} =
(\frac{1}{2},\ldots,\frac{1}{2})$).  Then we have the homotopy:
 $$F: \Ind_n \times [0,1] \rightarrow \Ind_n$$
 $$({\bf p},t) \mapsto (1-t){\bf p}+t{\bf x},$$
 for which $F(\cdot, 0)$ is the identity map, $F(\cdot, 1)$ maps
 $\Ind_n$ to ${\bf x}$, and $F({\bf p},t)\in \Ind_n$ for all $t\in
 [0,1]$ by Proposition~\ref{affine}.  

An early classical topological result of Smith~\cite{S} implies that
any subset $S$ of Euclidean space is not contractable if there is a 
continuous function $f:S \rightarrow S$ that has period two (i.e. $f \circ f$ is the identity map)
and which has no fixed point.    For $\GST_n$, the map ${\bf p}
   \mapsto {\bf 1} - {\bf p}$ is such a function, and since $\GST_n$ is a subset of
 Euclidean space it follows  that $\GST_n$ is not contractable. 
 \end{proof}

\subsection{Connectedness of $\GST_n$}
Since $\Ind_n$ is contractable, it is connected.  The connectedness of
$\GST_n$ is much more subtle and depends on the eigenvalues of the
Hessian matrix $H_n$ of $\psi$.  Consider any two points ${\bf p},
{\bf q} \in \GST_n$. By Proposition~\ref{allLines},  there are
 straight-line-paths from ${\bf p}$ to ${\bf m} = (\frac{1}{2},
 \ldots, \frac{1}{2})$, and from ${\bf m}$ to ${\bf q}$ and the
 concatenation of these two paths lies entirely in $\Ind_n$. However,
 exactly one point on this concatenated path, namely ${\bf m}$, fails to lie in
 $\Inf_n$.  It is not enough to show there is a
 `perturbed' path within $\Ind_n$ from ${\bf p}$ to ${\bf q}$  that
 avoids ${\bf m}$; we must also avoid all points not in $\Inf_n$.  
To study this further, we require one more topological result.  
\begin{lem}
\label{topology}
Let $M$ be a compact manifold and $I$ an open interval. Let ${\bf p} = ({\bf
  x},t)\in M\times I$ and
${\bf q} = ({\bf y},s)\in M\times I$. Then there exists
$\phi:M\rightarrow M\times I$ such that $M$ is homeomorphic to 
$\image(\phi)$ and ${\bf p}, {\bf q}\in\image(\phi)$.  
\end{lem}

\begin{proof}
 Let $f:M\rightarrow I$ be any
continuous function such that $f({\bf x}) = t$ and $f({\bf y}) = s$.
Set $\phi: M\rightarrow M\times I$ to be $\phi({\bf v}) = ({\bf v},
f({\bf v}))$ for any ${\bf v}\in M$.  By construction, $\phi$ is
continuous, since $f$ is continuous. It is one-to-one, since it is
the identity on the first coordinate of the image.  Since $M$
is compact, $M\times I$ is Hausdorff and $\phi$ is continuous and
one-to-one, $\phi^{-1}$ is also continuous~\cite[Corollary 5.9.2]{Su}.
Hence $M$ is homeomorphic to the image of  $\phi$.   
\end{proof}

 \begin{thm}
\label{gstConnected}
If the quadratic form $Q_n$ (equivalently, the Hessian matrix $H_n$)
has at least two strictly positive and at least two strictly negative
eigenvalues and $n\geq 8$, then $\GST_n$ is connected.  If $H_n$ has
only one strictly positive or one strictly negative eigenvalue, then
$\GST_n$ is disconnected.
 \end{thm}
 \begin{proof}
 Let $\Inf_n^c$ denote the 
 linear subspace of $\RR^n$ of dimension $\lceil n/2 \rceil$ defined by:
$$x_i - x_{n-i+1}=0 \mbox{  for all } i \in [n]. $$

 Consider any two points ${\bf p}, {\bf q} \in \GST_n$.  We first show
 that, if $n\geq 8$ and $Q_n$ has certain eigenvalues, there is a path
 from ${\bf p}$ to ${\bf q}$ that lies entirely in 
 $\GST_n$. We then use related structures to argue that if $Q_n$ has
 exactly one strictly postive or strictly negative eigenvalue then
 $\GST_n$ is disconnected.  

 Since ${\bf m} = \frac{1}{2}{\bf 1}$,
 Proposition~\ref{affine} (or Taylor expansion using the fact that
 ${\bf m}$ is a zero of $\psi$ and a critical point) implies that 
$$\psi({\bf m} + {\bf x}) = {\bf x}^T Q_n {\bf x},$$
  where $Q_n$  is the matrix corresponding to the quadratic form $\psi$
  (see Section~\ref{sec:quadForm} and Remark~\ref{diagonalize}). 
Let $P$ and $D$ be as in Remark~\ref{diagonalize}, (i.e.  $P^TQ_nP = D$, 
where $D$ is the diagonal matrix of real eigenvalues of $Q$ and $P$ is
a real orthogonal matrix).  Let ${\bf y} = P^T{\bf x}$ (so ${\bf x} =
P{\bf y}$). We then have:
\begin{equation}
\label{center}
\psi({\bf m}+ {\bf x}) = {\bf x}^T Q_n {\bf x} = {\bf y}^T P^T  Q_n P {\bf y} = {\bf y}^T D {\bf y}.
\end{equation}

For our argument, we need a few subsets of $\RR^n$ which depend on $D$
and $P$.  The first two are 
$$T_1 = \{P^T{\bf x}: {\bf x} \in \Inf_n^c\},$$
and
$$T_2 = \{P^T{\bf x}: {\bf x} \in [-1/3, 1/3]^n\}.$$
Since $P$ has full rank, it follows that $T_1$ is a linear subspace of
$\RR^n$ of dimension $\lceil n/2 \rceil$, while  
$T_2$ is a convex polytope of dimension $n$, containing ${\bf 0}$.
The others are defined in the next paragraph.  
 
By Proposition~\ref{eigenspace}, $D$ has zero as an eigenvalue with
geometric multiplicity one.   
Suppose that $D$ has $k$ strictly positive eigenvalues, and $l$
strictly negative eigenvalues, so that $k + l + 1 = n$.  By
Theorem~\ref{eigenvalueNumbers}, $k>0$ and $l>0$.   We may assume that the first
eigenvalue is $0$ and that the next $k$ eigenvalues $\lambda_1, \ldots, \lambda_k$ are all
strictly positive, while the final $l$ eigenvalues, $\mu_1, \ldots,
\mu_l$ are all strictly negative.  For any $s>0$ and $t \geq 0$, the set 
$$S_{s,t}:=   \{{\bf y} \in \RR^n: -s < y_1 < s , \sum_{i=1}^k  \lambda_iy_i^2 = t
\mbox{ and } \sum_{j=1}^l (-\mu_j) y^2_{k+j}= t \}$$ 
is a set of solutions to the equation:
$${\bf y}^TD{\bf y} = 0.$$  

Observe that we have the homeomorphism $S_{s,t} \cong  I\times S^{k-1}\times S^{l-1}$.  If
$\min\{k,l\}> 1$, then $S_{s,t}$ is the cross product of an open
interval -- call it $I_s$ -- and a compact orientable $m =
(n-3)$-manifold which we denote by $M_t$.   However, if $\min\{k,l\}
= 1$ (say $k = 1$, so $l = n-2$)  then $S_{s,t}$ is the cross product
of the following three spaces:  a open interval $I_s$, two points
(i.e. $S^0$,  which comes from the equation $\lambda_1 y_2^2 = t$) and
an $(n-3)$-sphere.   

We first assume that $k,l>1$ and $n\geq 8$, and continue the proof
that $\GST_n$ is connected.  We then look at what happens if
$\min\{k,l\} = 1$ and argue that $\GST_n$ is disconnected.
Assume that $k,l>1$.

Set $s,t'>0 $ sufficiently small so that $S_{s,t'} \subseteq T_2$
(the requirement that $S_{s,t'} \subseteq T_2$ is so that ${\bf m} +
{\bf x}$ for ${\bf x} \in S_{s,t'}$ lies in $[0,1]^n$, which is a
requirement of Proposition \ref{affine}(ii) for ${\bf m} + {\bf x}$ to
be in $\Ind_n$).  Let ${\bf y_p} = c_1P^T {\bf p}$ and ${\bf y_q} = c_2P^T{\bf q}$ where $c_1>0$ and $c_2>0$ are chosen sufficiently small to ensure that, for some $t \in (0,t']$, we have:
${\bf y_p}, {\bf y_q} \in S_{s,t}$.

Write ${\bf y_p} = (u_p,{\bf u_p})$ and $
{\bf y_q} = (v_q,{\bf v_q})$.  By Lemma~\ref{topology}, there exists
$\phi:M_t \rightarrow I_s \times M_t  = S_{t,s}$ such that $\phi({\bf
  u_p}) = {\bf y_p}$ and $\phi({\bf v_q}) = {\bf y_q}$ and $M_t$ is
homemorphic to the image of $\phi$.  For ease of notation and
acknowledging the abuse, we set $M_t = \image(\phi)$.  

Set $A = M_t\cap T_1$. Thus $A$ is a closed and
bounded subspace of $\RR^{\lceil n/2 \rceil}$. Therefore, $A$ is a
proper closed subset of $M_t$ as 
long as $m = n-3 > \lceil n/2\rceil$, which is true for $n\geq
8$. In addition, $A$ is locally contractable (it is a CW-complex). 

In the following discussion, we compute all homology modules over
$\ZZ$.  Consider the terminal end of the 
long exact sequence relating homology to relative homology:
\begin{equation}
\label{exact}
\cdots\rightarrow H_1(M_t, M_t - A) \rightarrow H_0(M_t -
A)\rightarrow H_0 (M_t) \rightarrow H_0(M_t, M_t-A) \rightarrow 0.
\end{equation} 
By Alexander Duality~\cite[Proposition 3.46]{H} we have:
$$H_i(M_t, M_t-A)\cong  H^{m-i}(A).$$
Therefore,
$$H_1(M_t, M_t-A)\cong H^{m-1}(A) \mbox{ and } H_0(M_t, M_t-A)\cong H^m(A).$$

For $t >0 $, ${\bf 0}\notin M_t$ and therefore ${\bf 0}\notin A$.
However, ${\bf 0}\in \RR^{\lceil n/2 \rceil}$, so $A$ is a proper closed subset of
$\RR^{\lceil n/2 \rceil}$ and hence it is a proper closed subspace of
a compact manifold (sphere) of 
dimension  $\lceil n/2 \rceil$ as well.  Since $\lceil n/2 \rceil\leq
m-1$ for $n\geq 8$,   
by ~\cite[Proposition 6.5]{M}, $H^{m-1}(A) = H^m(A) = 0$  (we are using
that $A$ is a CW-complex so \u Cech cohomology coincides with singular
cohomology).  Hence 
the exactness of the sequence in (\ref{exact}) implies $$H_0(M_t -
A)\cong H_0 (M_t) \cong \ZZ.$$  Therefore, $M_t - A$ is connected.

By the connectivity of $M_t - A$ and the fact that $\phi$ is a
homeomorphism, there is a path in $\GST_n$ from ${\bf m} + P{\bf y_p}$
to ${\bf m} + P{\bf y_q}$.  We can then sandwich this path between the
straight-line-paths from ${\bf p}$ to ${\bf m} + P{\bf y_p}$ and from
${\bf m } + P{\bf y_q}$ to ${\bf q}$ (which are in $[0,1]^n$ for
sufficiently small $c_1, c_2$ and in $\GST_n$ by
Theorem~\ref{affine} and Remark~\ref{equivalence}) to obtain the required path in $\GST_n$ 
from ${\bf p}$ to ${\bf q}$. 

Now assume that $\min\{k,l\} = 1$.   Without loss of generality, take $k
= 1$ and hence $l = n-2$.  In this case, $S_{s,t} \cong I_s\times S^0\times S^{n-3}$.  
The $S^0$ consists of the two points which come from the equation $\lambda_1
y_2^2 = t$.  Thus we see that $S_{s,t}$ for $t >0$ is two copies of $I_s\times
S^{n-3}$, each located in the $y_2$ coordinate at the values $\pm
\sqrt{\frac{t}{\lambda_1}}$. Let $S^1_{s,t}$ and $S^2_{s,t}$ denote the
two copies of $I_s\times
S^{n-3}$ for a given $s$ and $t$.  Then for $t >0$ and all $s \geq 0$, these spaces are
disconnected and therefore the union over all $s\geq 0$, $t>0$ of
$S^1_{s,t}$ is disconnected from the union over all $s\geq 0$, $t>0$
of $S^2_{s,t}$.  The union of the spaces $S_{s,t}$ over all $s,t\geq 0$ is the set of all
solutions to ${\bf y}^TD{\bf y} = 0$, which is connected by  joining
$S^1_{s,t}$ and $S^2_{s,t}$ in the shared space $S^1_{s,0}
= S^2_{s,0}$.  However, $S^1_{s,0} = S^2_{s,0}$ is all points of the form ${\bf y} = (y,0,\ldots,
0)^T$ and  $P{\bf y} = y(\frac{1}{\sqrt{n}}, \ldots,
\frac{1}{\sqrt{n}})^T$, which is in $\Inf_n^c$.  Thus $\GST_n$
is disconnected, since the set of solutions to ${\bf y}^TD{\bf y}=0$
includes $\GST_n$ (by Eqn. (\ref{center})), and we have shown that elements of $\GST_n$
lie in two disjoint components of this space. 
\end{proof}

\begin{cor}\label{cor:connected}
$\GST_n$ is disconnected for $n = 3,4$ and $\GST_n$ is connected for
  $n \geq 8 $.  
\end{cor}

\begin{proof}
Direct computation shows that $H_3$ has one positive and one negative
eigenvalue and $H_4$ has one positive
eigenvalue and two negative eigenvalues and thus $\GST_3$ and $\GST_4$
are disconnected  (of course we also know this for
$n = 3$ from direct computation given in
Section~\ref{smalln}). Theorem~\ref{eigenvalueNumbers} shows that for
$n\geq 8$, $H_n$ has at least two strictly negative and two strictly
positive eigenvalues and therefore $GST_n$ is connected.   
\end{proof}

\section{Concluding comments}
We consider it an interesting question to determine whether $\GST_n$ for $n
= 5, 6, 7$ is connected or disconnected.
Theorem~\ref{eigenvalueNumbers} implies 
that $H_n$ (equivalently, $Q_n$) has at
least two positive and two negative eigenvalues for $n = 6,7$.
Direct computation shows the same is true for $n = 5$,  
but the dimensions of $M_t$ and $A$ do not suffice for the homology
argument given in the proof of Theorem~\ref{gstConnected}.

Further exploration of the topology of $\GST_n$ may be of interest, for
example classification up to homotopy or homeomorphism.  Also, note
that the (two) connected components of $\GST_3$ are contractable, and we
leave this question open for $n = 4$ (and $5\leq n\leq 7$, if they are
disconnected).  

We gave a thorough analysis of the GST set-up where $r = \frac{1}{2}$
and $p_k = q_k$.  One possible approach to the study of the
probabilities where influence and independence collide for more general values of
$r$, $p_k$,  and $q_k$ might be to treat $r$, $p_k$, $q_k$ as
variables in a ring $R = k[r, p_1, \ldots, p_n, q_1, \ldots, q_n]$ and
use polynomial ring theory. From a practical point of view, the
flexibility to allow $r$ to vary seems interesting.

\section{Acknowledgments}
We thank the Burroughs Wellcome Fund Collaborative Research Travel Grant, and the
New Zealand Marsden Fund for funding.  


\section{Appendix}
We include here the details for the computations of
$D_{\lfloor\frac{n}{2}\rfloor + 1}$ from the end of
Section~\ref{sec:quadForm}. As in that section, we set $H = H_n$ to
clean up the notation.  

We need to argue that $D_{\lfloor \frac{n}{2}\rfloor + 1}
>0$. We recall a few of the formulae found in the proof of
Theorem~\ref{eigenvalues} since we use them all:
\begin{align*}
D_1 & = H_{11} = \frac{2}{N^2},\\
L_{i1} & = \binom{n-1}{i-1}\mbox{, for } 1\leq i\leq n-1,\\
A_{ij} & = H_{ij} = D_1L_{i1}L_{j1}, \mbox{ for all } i\neq n-j,\ n-j+1,\\
L_{ij} & = -\frac{1}{D_j}\bigg(\sum_{k=2}^{j-1}L_{ik}L_{jk}D_k\bigg) \mbox{ for all } i\neq n-j,\ n-j+1,\\
L_{ij} & = 0 \mbox{ for all } 1 < i,j \leq \lfloor
\frac{n}{2}\rfloor,\\
D_i & =  - \frac{2}{N}\binom{n-2}{i-2} < 0, \mbox{ for all } 1 <
i\leq\lfloor\frac{1}{2}\rfloor.
\end{align*}

We first assume that $n$ is odd, so that $\lfloor
\frac{n}{2}\rfloor + 1 + \lfloor \frac{n}{2}\rfloor = n$. For ease of
notation, let $t = \lfloor \frac{n}{2}\rfloor + 1$.  Then: 
$$L_{tt-1} = \frac{1}{D_{t-1}} \bigg(H_{tt-1} -
\sum_{k=1}^{t-2}L_{tk}L_{t-1k}D_k\bigg).$$ 
However, $L_{t-1k} = 0$ for $2\leq k\leq t-2 < \lfloor \frac{n}{2}\rfloor $ since 
$t-1 = \lfloor \frac{n}{2}\rfloor$.    
Using that $D_{t-1} = -\frac{2}{N}\binom{n-2}{t-3}$, we have:
\begin{align}\label{Ltt}
L_{tt-1} = &  -\frac{1}{\frac{2}{N}\binom{n-2}{t-3}} \bigg(\frac{2}{N^2}
\binom{n-1}{t-1}\binom{n-1}{t-2} - \frac{2}{N}\binom{n-2}{t-1} -
\binom{n-1}{t-1}\binom{n-1}{t-2}\frac{2}{N^2}\bigg)\\
 = & \frac{\binom{n-2}{t-1}}{\binom{n-2}{t-3}}\notag.
\end{align}
Therefore:
\begin{align}
D_t = & H_{tt} - \sum_{k = 1}^{t-1}L_{tk}^2D_k\notag\\
= & \frac{2}{N^2}\binom{n-1}{t-1}^2 - \frac{2}{N}\binom{n-2}{t-2}+\epsilon -
\binom{n-1}{t-1}^2\frac{2}{N^2} - L_{tt-1}^2D_{t-1} \notag\\
= & - \frac{2}{N}\binom{n-2}{t-2} +\epsilon -
\bigg(\frac{\binom{n-2}{t-1}}{\binom{n-2}{t-3}}\bigg)^2\bigg(-
\frac{2}{N}\binom{n-2}{t-3}\bigg)\notag\\
=&  \frac{2}{N}\bigg(-\binom{n-2}{t-2} +
\frac{\binom{n-2}{t-1}^2}{\binom{n-2}{t-3}}\bigg) + \epsilon\notag\\
= & \frac{2}{N}\bigg(-\binom{n-2}{\lfloor\frac{n}{2}\rfloor - 1} +
\bigg(\frac{\lfloor\frac{n}{2}\rfloor +
  1}{\lfloor\frac{n}{2}\rfloor-1}\bigg)\binom{n-2}{\lfloor\frac{n}{2}\rfloor}\bigg)+\epsilon\notag\\
 = &
 \frac{2}{N}\binom{n-2}{\lfloor\frac{n}{2}\rfloor}\bigg(\frac{2}{\lfloor\frac{n}{2}\rfloor-1}\bigg)+\epsilon
 > 0.\label{last}
\end{align}
where (\ref{last}) uses the symmetry of the binomial.  

Now assume $n$ is even, so that $\lfloor\frac{n}{2}\rfloor =
\frac{n}{2}$.  This time, let $t = \frac{n}{2}$.  Then the entries of
$L$ we need to be concerned with are $L_{t+1,t-1}$ and $L_{t+1,t}$.
In both cases, as in Eqn.~(\ref{Ltt}), the sum has all terms zero,
except for the first one.  We note that $\epsilon$ potentially
appears in $L_{t+1k}$, but $L_{tk}$ or
$L_{t-1k}$ are still zero and hence the full sum is zero.  Therefore:
$$L_{t+1,t}  =\frac{\epsilon}{D_t} = -\frac{\epsilon N}{2\binom{n-2}{t-2}},$$ and
$$L_{t+1, t-1} = \frac{\binom{n-2}{t}}{\binom{n-2}{t-3}}.$$
We are now ready to compute $D_{t+1}$.
\begin{align*}
D_{t+1} =& A_{t+1,t+1} - \sum_{k = 1}^t L_{t+1k}^2 D_k\\
= &  -\frac{2}{N}\binom{n-2}{t-1} - L_{t+1t-1}^2D_{t-1}-
L_{t+1t}D_t\\
= & -\frac{2}{N}\binom{n-2}{t-1} -
\bigg(\frac{\binom{n-2}{t}}{\binom{n-2}{t-3}}\bigg)^2\bigg(-\frac{2}{N}\binom{n-2}{t-3}\bigg)
- \bigg(-\frac{\epsilon
  N}{2\binom{n-2}{t-2}}\bigg)^2\bigg(-\frac{2}{N}\binom{n-2}{t-2}\bigg)\\
=& \frac{2}{N}\bigg(-\binom{n-2}{t-1} +
\frac{\binom{n-2}{t}^2}{\binom{n-2}{t-3}}\bigg) +
\frac{\epsilon^2N}{2\binom{n-2}{t-2}}\\
 = & \frac{2}{N}\bigg(-\binom{n-2}{t-1}
 +\frac{(\frac{n}{2}+1)(n-2)!}{(\frac{n}{2} -
   2)(\frac{n}{2})!(n-\frac{n}{2}-2)!} \bigg)+
   \frac{\epsilon^2N}{2\binom{n-2}{t-2}}\\
= & \frac{2}{N}\bigg(-\binom{n-2}{t-1}
 +\frac{(\frac{n}{2}+1)(n-2)!}{(\frac{n}{2} -
   2)(\frac{n}{2})!(\frac{n}{2}-2)!} \bigg)+
   \frac{\epsilon^2N}{2\binom{n-2}{t-2}}\\
= & \frac{2}{N}\bigg(-\binom{n-2}{t-1}
 +\frac{(\frac{n}{2}+1)(\frac{n}{2}-1)}{(\frac{n}{2})(\frac{n}{2}-2)}\binom{n-2}{t-1}\bigg) +
  \frac{\epsilon^2N}{2\binom{n-2}{t-2}}\\
 = & 
 \frac{2}{N}\binom{n-2}{t-1}\bigg(\frac{n-1}{\frac{n}{2}(\frac{n}{2}-2)}\bigg)
 + \frac{\epsilon^2N}{2\binom{n-2}{t-2}} >0.
\end{align*}

\end{document}